\documentclass{article}
\usepackage{amsthm}
\usepackage{amssymb}
\usepackage{amsmath}

\theoremstyle{plain}
\newtheorem{thm}{Theorem}[section]
\newtheorem{lemma}[thm]{Lemma}

\theoremstyle{definition}
\newtheorem{definition}[thm]{Definition}

\begin{document}
\title{The strong Gelfand pairs of the Suzuki family of groups}
\author{Joseph E. Marrow}



\maketitle

\begin{abstract}
We find every subgroup $H\leq Sz(q)$ so that the pair $(Sz(q), H)$ is a strong Gelfand pair.
\end{abstract}
\section{Introduction}
\begin{definition}
A finite group $G$ and a subgroup $H$, form a \emph{Gelfand pair} $(G, H)$ if the trivial character of $H$ induces a multiplicity-free character of $G$.
\end{definition}

There are a number of equivalent conditions to two groups being a Gelfand pair. These include, but are not limited to
\begin{itemize}
\item The centralizer algebra of the permutation representation is commutative;
\item The algebra of $(H, H)$-double invariant functions of $G$ with multipliciation defined by convolution is commutative;
\item The double cosets $HgH$ forms a commutative Schur ring over $G$.
\end{itemize}
with these conditions and others as found in \cite{grady, totalArg, markov}.

This requirement can be strengthened to what is known as a strong Gelfand pair, and this is what we will be considering. Strong Gelfand pairs also have a number of equivalent conditions. Of particular note is the relationship to Schur rings, fundamental objects in algebraic combinatorics \cite{konig, srings, schur}; $(G, H)$ being a strong Gelfand pair  means that the Schur ring determined by the $H$-classes $g^H$ is a commutative ring \cite{symplectic}.

While the determination of the strong Gelfand pairs of various families of finite groups is an ongoing process, a number of strides have been made. When $G\in \{S_n$, $D_{2n}$, $Dic_{4n}$, $\mathrm{SL}_2(p^n)$, $\mathrm{Sp}_4(2^e)\}$ the strong Gelfand pairs $(G, H)$ are known \cite{grady, totalArg, symplectic, thesis}. The strong Gelfand pairs for other families of groups have also been studied \cite{wreath}. 

While \cite{grady} introduces the notion of an extra strong Gelfand pair, we do not consider this.

The Suzuki family of simple groups $Sz(2^{2n+1})$, described in \cite{suzuki}, has a number of useful constructions \cite{ree, tits, wilson}. This is an infinite family of groups of Lie type. They are maximal subgroups of the finite symplectic groups \cite{ono, MaxSub} and it has already been shown that $(\mathrm{Sp}_4(2^{2n+1}), Sz(2^{2n+1}))$ is not a strong Gelfand pair \cite{symplectic, thesis}. We now turn our attention to determining if the groups $Sz(q)$ contain any proper strong Gelfand subgroups.

\section{Preliminaries}
\begin{definition}
A \emph{strong Gelfand pair} $(G, H)$ is a pair of finite groups $H\leq G$ such that $\langle \chi\downarrow H, \psi\rangle\leq 1$ for all irreducible complex characters $\chi\in\hat{G}$, $\psi\in\hat{H}$. If $G, H)$ is a strong Gelfand pair, we call $H$ a \emph{strong Gelfand subgroup} of $G$. 
\end{definition}

In particular, $(G, G)$ is always a strong Gelfand pair, we consider this one trivial. We take a critical result from \cite{totalArg}, given here as Lemma \ref{stack}.
\begin{lemma}\label{stack}
Let $K \leq H \leq G$ be finite groups. If $(G, H)$ is not a strong Gelfand pair, then $(G, K)$ is not a strong Gelfand pair.
\end{lemma}

Using Lemma \ref{stack}, we will be able to determine the strong Gelfand pairs $(Sz(q), H)$ of the Suzuki family of groups by initially focusing on the maximal subgroups $H\leq Sz(q)$. In order to show there are no nontrivial strong Gelfand pairs of $Sz(q)$, $q=2^{2n+1}$, $n\geq1$, we will need to know three things. Firstly, that the maximal subgroups of $Sz(q)$ are as follows, taken from \cite{MaxSub}. Let $q=2^e$, $e>1$, $e$ odd. Note $Sz(2)\cong F_{20}\cong 5\colon 4$. 
\begin{table}[h!]\label{max}
\centering
\begin{tabular}{cc}
Group & Condition\\
\hline
$E_q^{1+1}\colon \mathcal{C}_{q-1}$ & \\
$\mathrm{D}_{2(q-1)}$ & \\
$(q-\sqrt{2q}+1)\colon 4$ & \\
$(q+\sqrt{2q}+1)\colon 4$ & \\
$Sz(q_0)$ & $q=q_0^r$, $r$ a prime, $q_0\neq2$.
\end{tabular}
\end{table}

Second, the degree of a character of $Sz(q)$, found in Lemma \ref{degree}, taken from \cite{suzuki}
\begin{lemma}\label{degree}
The Suzuki group has a character of degree $q^2+1$. 
\end{lemma}

Lastly, we will use the result from \cite{totalArg} that shows a condition under which $(G, H)$ cannot form a strong Gelfand pair. 
\begin{definition}
The \emph{total character} of a group $H$ is the sum of the irreducible characters of $H$. 
\end{definition}

\begin{lemma}\label{total}
Let $H\leq G$ and $\tau$ be the total character of $H$ and $\chi$ an irreducible character of $G$. If $\deg(\tau) < \deg(\chi)$, then $(G, H)$ is not a strong Gelfand pair.
\end{lemma}

\section{Main Result}

\subsection{$E_q^{1+1}\colon \mathcal{C}_{q-1}$}
\begin{lemma}\label{e}
Let $q>2$ be an odd power of $2$. The pair $(Sz(q), E_q^{1+1}\colon \mathcal{C}_{q-1})$ is not a strong Gelfand pair.
\end{lemma}
\begin{proof}
Let $q=2^{2n+1}$. Let $H=S\colon M$, where $S=E_q^{1+1}$ and $M=\mathcal{C}_{q-1}$. We provide the following excerpt from \cite{french}, translated by the author. 

\textit{The subgroup $H$ has $q+2$ conjugacy classes, namely the $q-1$ classes containing the elements of $M$, the class of involutions, and the two classes of order $4$ elements.\\
So $H$ has $3$ nonlinear characters (of degree at least two) which we will notate $\mu$, $\partial_1$, $\partial_2$.\\
-- $\mu$ is the doubly transitive character of degree $q-1$ obtained from the doubly transitive representation of $H/Z(S)$.\\
-- $H$ has two real conjugate classes, the class of the identity element and the class of involutions.\\
So $H$ has only $2$ real characters: $1_H$ and $\mu$. Consequently, $\partial_1=\overline{\partial_2}$. Furthermore, we have:
$$
|H| = q^2(q-1) = (q-1)+(q-1)^2 + 2(\partial_1(1))^2,
$$
hence $\partial_1(1)=(q-1)\left(\frac{q}{2}\right)^\frac{1}{2}$.
}

We will write $\sqrt{\frac{q}{2}} = 2^n$. Then we have $2(q-1)+2^n(q-1)$ as the degree of the total character of $E_q^{1+1}\colon \mathcal{C}_{q-1}$. The result then follows from Lemma \ref{degree} and Lemma \ref{total}.
\end{proof}

\subsection{$\mathrm{D}_{2(q-1)}$ and $(q\pm\sqrt{2q}+1)\colon 4$}
\begin{lemma}\label{most}
Let $q>2$ be an odd power of $2$. None of $(Sz(q), \mathrm{D}_{2(q-1)})$, $(Sz(q), (q+\sqrt{2q}+1)\colon 4)$, or $(Sz(q), (q-\sqrt{2q}+1)\colon 4)$ are strong Gelfand pairs.
\end{lemma}
\begin{proof}
The orders of the subgroups in question are $2(q-1)$, $4(q+\sqrt{2q}+1)$ and $4(q-\sqrt{2q}+1)$. Hence we are done by Lemma \ref{degree} and \ref{total}.
\end{proof}

\subsection{$Sz(q_0)$}
\begin{lemma}\label{sz}
Let $q=2^{2n+1} = q_0^r$ for $r$ a prime and $q_0\neq 2$. The pair $(Sz(q), Sz(q_0))$ is not a strong Gelfand pair.
\end{lemma}
\begin{proof}
For ease let $q_0=2^{2k+1}$. The degree of the total character of $Sz(q_0)$ is $2^{k+1}(q_0-1)-q_0(q_0-1)+q_0^3$, as seen in \cite{suzuki}. This shows that
$$
q^2+1 = q_0^{2r}+1 \geq q_0^4+1 \geq 2^{k+1}(q_0-1)-q_0(q_0-1)+q_0^3.
$$
This gives the result by Lemma \ref{degree} and Lemma \ref{total}.
\end{proof}

\begin{thm}\label{main}
For $q>2$, the group $Sz(q)$ has no nontrivial strong Gelfand pairs.
\end{thm}
\begin{proof}
We see from Lemmas \ref{e}, \ref{most}, and \ref{sz} that no maximal subgroup is a strong Gelfand subgroup. By Lemma \ref{stack}, since no maximal subgroups are strong Gelfand subgroups, no proper subgroup is a strong Gelfand subgroup.
\end{proof}

Lastly, for the discarded case, Theorem \ref{disc}.

\begin{thm}\label{disc}
The group $Sz(2)$ has four strong Gelfand subgroups; namely $Sz(2)$, $\mathrm{D}_{10}$, $\mathcal{C}_5$, and $\mathcal{C}_4$.
\end{thm}
\begin{proof}
This is easily verified electronically.
\end{proof}

\subsection*{Ethical approval}
Not applicable.

\subsection*{Competing interests} 
Not applicable.

\subsection*{Authors' contributions} 
Not applicable

\subsection*{Availability of data and materials}
Not applicable. 
\smallskip

\subsection*{Funding}
Not applicable.


\begin{thebibliography}{40}
\bibitem{grady}
Anderson, Gradin; Humphries, Stephen P.; Nicholson, Nathan Strong Gelfand pairs of symmetric groups. J. Algebra Appl. 20 (2021), no. 4, Paper No. 2150054, 22 pp.
\bibitem{totalArg}
Barton, Andrea; Humphries, Stephen Strong Gelfand Pairs of SL(2, p). J. Algebra Appl. 22
(2023), no. 6, Paper No. 2350133, 13 pp. https://doi.org/10.1142/S0219498823501335
\bibitem{MaxSub}
  Bray, J., Holt, D., Roney-Dougal, C.:
  The Maximal Subgroups of the Low-Dimensional Finite Classical Groups
  Cambridge University Press, Cambridge (2013)
\bibitem{wreath}
Can, Mahir Bilen; She, Yiyang; Speyer, Liron Strong Gelfand subgroups of $F\wr S_n$. International J. of Mathematics, 32(02) (2021).
\bibitem{markov}
Ceccherini-Silberstein, T.; Scarabotti, F.; Tolli, F., Harmonic analysis on finite groups, in
Representation Theory, Gelfand Pairs and Markov Chains, Vol. 108 Cambridge Studies in
Advanced Mathematics, Cambridge University Press, Cambridge, 2008, xiv + 440 pp.
\bibitem{symplectic}
Humphries, S., Marrow, J.: The strong Gelfand pairs of $\mathrm{Sp}_2(q)$ for $q$ even.  \emph{Mathematics} 2025, 13(18), 2977; https://doi.org/10.3390/math13182977 
\bibitem{schur}
I. Schur
Zur Theorie der einfach transitiven Permutationgruppen
S.-B. Preus Akad. Wiss. Phys.-Math. Kl. (1933), pp. 598-623
\bibitem{konig}
M.Kh. Klin, R. Pöschel
The König problem, the isomorphism problem for cyclic graphs and the method of Schur rings
Algebraic Methods in Graph Theory, vol. 1, 2 (Szeged, 1978), Colloq. Math. Soc. János Bolyai, vol. 25, North-Holland, Amsterdam, New York (1981), pp. 405-434
\bibitem{srings}
M.Ch. Klin, N.L. Najmark, R. Pöschel, Schur rings over Akad. der Wiss. der DDR Inst. für Math., Preprint P-MATH- 14/81 (1981), 1–30
\bibitem{thesis}
Marrow, J.;
Strong Gelfand Pairs of Some Finite Groups. BYU (2022) Master's Thesis.
\bibitem{french}
   Nouacer, Z.:
  Caractr\`{e}res et sous-groupes des groupes de Suzuki.
  Diagrammes, \textbf{8} (1982) (French)
\bibitem{ono}
Ono, Takashi (1962), "An identification of Suzuki groups with groups of generalized Lie type.", Annals of Mathematics, Second Series, 75 (2): 251–259, doi:10.2307/1970173, ISSN 0003-486X, JSTOR 1970173, MR 0132780
\bibitem{ree}
Steinberg, Robert (1959), "Variations on a theme of Chevalley", Pacific Journal of Mathematics, 9 (3): 875–891, doi:10.2140/pjm.1959.9.875, ISSN 0030-8730, MR 0109191
\bibitem{suzuki}
Suzuki, Michio (1960), "A new type of simple groups of finite order", Proceedings of the National Academy of Sciences of the United States of America, 46 (6): 868–870, Bibcode:1960PNAS...46..868S, doi:10.1073/pnas.46.6.868, ISSN 0027-8424, JSTOR 70960, MR 0120283, PMC 222949, PMID 16590684
\bibitem{tits}
Tits, Jacques (1962), "Ovoïdes et groupes de Suzuki", Archiv der Mathematik, 13: 187–198, doi:10.1007/BF01650065, ISSN 0003-9268, MR 0140572, S2CID 121482873
\bibitem{wilson}
Wilson, Robert A. (2010), "A new approach to the Suzuki groups", Mathematical Proceedings of the Cambridge Philosophical Society, 148 (3): 425–428, doi:10.1017/S0305004109990399, ISSN 0305-0041, MR 2609300, S2CID 18046565

\end{thebibliography}
\end{document}